\documentclass{amsart}

\usepackage{amssymb, amsmath}
\usepackage{mathrsfs}
\usepackage{amscd}
\usepackage[active]{srcltx}
\usepackage{verbatim}

\usepackage[colorlinks,linkcolor={blue},citecolor={blue},urlcolor={red},]{hyperref}

\usepackage[fleqn,tbtags]{mathtools}

\usepackage{newsymbol}

\let\mathcal \undefined
\def\mathcal{\mathscr}

\let\emptyset \undefined
\let\ge       \undefined
\let\le       \undefined
\newsymbol\le          1336  \let\leq\le
\newsymbol\ge          133E  \let\geq\ge
\newsymbol\emptyset    203F
\newsymbol\notle       230A
\newsymbol\notge       230B

\theoremstyle{plain}
\newtheorem{theorem}{Theorem}[section]
\theoremstyle{remark}
\newtheorem{remark}[theorem]{Remark}

\theoremstyle{plain}
\newtheorem{corollary}[theorem]{Corollary}
\newtheorem{lemma}[theorem]{Lemma}
\newtheorem{proposition}[theorem]{Proposition}

\numberwithin{equation}{section}

\def\N{{\mathbb N}}
\def\Z{{\mathbb Z}}

\def\R{{\mathbb R}}
\def\C{{\mathbb C}}

\renewcommand{\P}{{\mathbb P}}

\def\om{\omega}

\newcommand{\eps}{\varepsilon}
\newcommand{\calL}{\mathscr{L}}
\newcommand{\n}{\Vert}

\newcommand{\Dom}{{\mathsf{D}}}

\newcommand{\wt}{\widetilde}

\newcommand{\ud}{\,{\rm d}}

\allowdisplaybreaks

\begin{document}

\title[Harmonic oscillators on UMD lattices]{Spectral multiplier theorems for abstract harmonic oscillators on UMD lattices}

\author{Jan van Neerven}

\address{Delft Institute of Applied Mathematics\\
Delft University of Technology\\P.O. Box 5031\\2600 GA Delft\\The Netherlands}
\email{J.M.A.M.vanNeerven@tudelft.nl}

\author{Pierre Portal}
\address{The Australian National University, Mathematical Sciences Institute, Hanna Neumann Building, Ngunnawal and Ngambri Country,
Canberra ACT 2601, Australia}
\email{Pierre.Portal@anu.edu.au}

\author{Himani Sharma}
\address{The Australian National University, Mathematical Sciences Institute, Hanna Neumann Building, Ngunnawal and Ngambri Country, Canberra ACT 2601, Australia.}
\email{Himani.Sharma@anu.edu.au}

\date{\today}


\thanks{The authors gratefully acknowledge financial support from the ARC Discovery Project DP160100941.}

\begin{abstract}
We consider operators acting on a UMD Banach lattice $X$ that have the same algebraic structure as the  position and momentum operators associated with the
 harmonic oscillator $-\frac12\Delta + \frac12|x|^{2} $ acting on $L^{2}(\mathbb{R}^{d})$. More precisely,
 we consider abstract harmonic oscillators
 of the form  $\frac12 \sum _{j=1} ^{d}(A_{j}^{2}+B_{j}^{2})$ for tuples of operators $A=(A_{j})_{j=1} ^{d}$ and $B=(B_{k})_{k=1} ^{d}$,
 where $iA_j$ and $iB_k$ are assumed to
generate $C_{0}$ groups and to satisfy the
 canonical commutator relations. We prove functional calculus results for these abstract
harmonic oscillators
that match classical H\"ormander spectral multiplier estimates for the harmonic oscillator  $-\frac12\Delta + \frac12|x|^{2}$   on $L^{p}(\mathbb{R}^{d})$. This covers situations where the underlying metric measure space is not doubling and the use of  function spaces that are not particularly well suited to extrapolation arguments.
For instance, as an application we treat the harmonic oscillator on mixed norm Bargmann-Fock spaces. Our approach is based on a transference principle for the Schr\"odinger representation of the Heisenberg group that allows us to reduce the problem to the study of the twisted Laplacian on the Bochner spaces $L^{2}(\R^{2d};X)$. This can be seen as a generalisation of the Stone--von Neumann theorem to UMD lattices $X$ that are not Hilbert spaces.
\end{abstract}

\maketitle

\section{Introduction}
The harmonic oscillator $$L_{HO} =-\frac12\Delta +\frac12|x|^{2} $$ is a self-adjoint operator on $L^{2}(\mathbb{R}^{d})$ whose functional calculus satisfies Mihlin-H\"ormander spectral multiplier estimates of the form
$$
\|f({ L_{HO}}) u\|_{p} \lesssim \underset{k=1,\dots,\lfloor \frac{d}{2} \rfloor+1}{\max} \|x\mapsto x^{k}f^{(k)}(x)\|_{\infty} \|u\|_{p}$$ for all $
f \in C^{\lfloor \frac{d}{2} \rfloor+1}(\mathbb{R}_{+})$, $p \in (1,\infty)$, and $u \in L^{p}(\R^{d}) {\cap L^2(\R^d)}$ \cite{t,tbook}. More generally,  such a spectral multiplier theorem holds for operators $L$ that are self-adjoint on $L^{2}(M)$ and such that $-L$ generates a $C_0$-semigroup $(e^{-tL})_{t\ge 0}$ satisfying appropriate heat kernel bounds on $L^{2}(M)$, as long as $M$ is a doubling metric measure space \cite{dsy}.
In this paper we consider operators $L$ acting on an abstract UMD lattice $X$, assuming that $L$ has the same algebraic structure as the harmonic oscillator. Examples of UMD lattices include the spaces  $L^{p}(M)$ with $1<p<\infty$ for a general measure spaces $M$ as well as the Bargmann-Fock spaces $A^{p,q}(\C^d)$ with $1<p,q<\infty$).  Specialising to the case $X=L^{p}(M)$ with $1<p<\infty$ for the moment, we will assume that $L$ is of the form
\begin{align}\label{eq:L} L = \frac{1}{2}\sum _{j=1} ^{d}(A_{j}^{2}+B_{j}^{2})
\end{align}
for a pair of $d$-tuples $A=(A_1,\dots,A_d)$ and $B = (B_1,\dots,B_d)$ of closed and densely defined operators acting on both $L^2(M)$ and $L^{p}(M)$ such that the following assumptions hold (our statements are informal; precise versions of the assumptions will be given below):
\begin{enumerate}
\item\label{it:conditions-Lq1}
the operators $A_{j}$ and $B_{j}$ are self-adjoint on $L^{2}(M)$ for all $j=1,\dots,d$.
\item\label{it:conditions-Lq2} on both $L^2(M)$ and $L^{p}(M)$, the operators $iA_j$ and $iB_j$ generate uniformly
bounded $C_0$-groups $(e^{itA_j})_{t\in \R}$ and $(e^{itB_j})_{t\in \R}$, respectively;
\item\label{it:conditions-Lq3} the {\em Weyl commutation relations} hold for all $s,t\in \R$ and $1\le j,k\le d$:
\begin{align*}
e^{isA_j}e^{itA_k} &= e^{itA_k}e^{isA_j}, \quad e^{isB_j}e^{itB_k} = e^{itB_k}e^{isB_j},\\
      e^{isA_j}e^{itB_k} & = e^{-ist \delta_{jk}} e^{itB_k}e^{isA_j}.
\end{align*}
Here, $\delta_{jk}$ is the usual Kronecker symbol.
\end{enumerate}
By the Stone--von Neumann uniqueness theorem (see, for example, Chapter 14 of \cite{Hall}), on $L^{2}(M)$,
the study of such an operator $L$ can be reduced to the study of
the harmonic oscillator $L_{HO}$ on $L^{2}(\mathbb{R}^{d})$ (see, e.g., the proof of Lemma \ref{lem:transfer} below). This by itself, however, does not imply spectral multiplier results in $L^{p}(\R^d)$ for exponents $q \neq 2$. A striking example is given by the shifted Ornstein-Uhlenbeck operator
$$L_{OU, \frac{d}{2}}=-\Delta +x\nabla + \frac12 d$$ on $L^{2}(\mathbb{R}^{d},\gamma)$ where $\gamma$ denotes the standard Gaussian measure on $\R^d$. { By \cite[Theorem 3.1]{NP1} and \cite[Theorem 5.2]{NP2}, this operator is of the form \eqref{eq:L}.}
For
$p \neq 2$, $f({ L_{OU,\frac{d}{2}}})$ can only be bounded in $L^{p}(\mathbb{R}^{d},\gamma)$ for functions $f$ that have an holomorphic extension to a (shifted) sector of the complex plane; see \cite{GMMST}. This example is extreme, however, in that the corresponding momentum operators $B_{j}$ do not even generate $C_0$-groups on $L^{p}(\mathbb{R}^{d},\gamma)$. Less extreme examples are given in \cite[Section 3]{NP2}, along with a conditional spectral multiplier theorem \cite[Theorem 8.5]{NP2}, that is, a result under an {\em a priori} boundedness assumption on the joint (Weyl) functional calculus of $(A,B)$.

The main contribution of this paper is that we are able to completely remove this {\em a priori} assumption.
To do so, we use the non-commutative Weyl functional calculus of $(A,B)$ as a source of inspiration rather than a technical tool, as is the case in recent work regarding the Ornstein-Uhlenbeck operator; see \cite{harris, NP1}.
Our proof starts with an application of the transference principle \cite[Proposition 6.3, Lemma 6.4]{NP2}
\begin{equation}\label{eq:transference}
\|f(L)\|_{\calL(L^p(M))} \lesssim_{p,q} \|f(\widetilde{L} \otimes I_{X})\|_{\calL(L^{2}(\mathbb{R}^{2d};L^p(M)))},
\end{equation}
where $p\in [1,\infty)$ and $\widetilde{L}$ is the {\em twisted Laplacian} on $L^p(\R^{2d})$ defined by
\begin{align}\label{eq:wtL} \widetilde{L} := \frac{1}{2}\sum  _{j=1} ^{d} (\widetilde Q_{j}^{2}+\widetilde P_{j}^{2});
\end{align}
here
$\widetilde Q_{j} = -\frac{1}{2} Q_{2,j}- P_{1,j}$ and $\widetilde P_{j}=\frac{1}{2} Q_{1,j}- P_{2,j}$, where
the self-adjoint operators $Q_{1,j}, Q_{2,j}$ and $P_{1,j}, P_{2,j}$ on $L^2(\R^{2d})$ are given by
\begin{equation}\label{eq:wtPQ}
\begin{aligned}
    & Q_{1,j}f(x,y):=x_jf(x,y),\; && Q_{2,j}f(x,y):=y_jf(x,y),\\
    &\, P_{1,j}f(x,y):=\frac{1}{i}\frac{\partial f}{\partial x_j}(x,y),\; && \, P_{2,j}f(x,y):=\frac{1}{i}\frac{\partial f}{\partial y_j}(x,y),
\end{aligned}
\end{equation}
 for $j=1,\dots,d$. The rigorous definition of the operators
$-\frac{1}{2} Q_{2,j}- P_{1,j}$ and $\frac{1}{2} Q_{1,j}- P_{2,j}$ as densely defined closed operators on $L^p(M)$
is provided by \eqref{eq:BCH2} below.

This result allows us to transfer the problem from $L$ acting on $L^{p}(M)$ to $\widetilde{L}\otimes I_{L^{p}(M)}$ acting on $L^{2}(\R^{2d};L^{p}(M))$. The transference principle we use is different from many transference results
available in the literature (such as the fundamental result of Coifman-Weiss \cite{cw}), in that it is based on a non-abelian group (the Heisenberg group).
Nevertheless, its proof is similar to abelian analogues and follows the general approach
presented in \cite{Haase-trans}, with twisted convolutions replacing convolutions to account for the non-abelian structure.
Reducing the study of $L$ to the study of $\widetilde{L}\otimes I_{L^{p}(M)}$
can be seen as an $L^p$ analogue of the Stone--von Neumann theorem, but one should note that it involves the  twisted Laplacian $\widetilde{L}$ rather than the standard harmonic oscillator.
The transference only shows that results about the functional calculus of
$\widetilde{L}\otimes I_{L^p(M)}$ acting on $L^{2}(\R^{2d};L^p(M))$ can be transferred to arbitrary abstract harmonic oscillators acting on $L^p(M)$.
The operator
$f(\widetilde{L})$ is not a standard Calder\'on-Zygmund operator on $L^{2}(\R^{2d})$; rather, it is
a special kind of oscillatory singular integral operator.

To estimate the norm $\|f(\widetilde{L})\|_{L^{p}(\R^{2d})}$,
we thus use the general spectral multiplier theorem \cite[Theorem 3.2]{dsy}
valid for operators $L'$ that are self-adjoint on $L^{2}(\R^{2d})$ and such that the semigroup $(\exp(-tL))_{t \geq 0}$ satisfies Gaussian heat kernel bounds (the latter assumption being proven using an explicit computation of the heat kernel from \cite{NP2}). Since this scalar-valued estimate also holds with Muckenhoupt weights, Rubio de Francia's extrapolation principle from \cite{rubio} then allows us to prove the required estimate on the norm $\|f(\widetilde{L} \otimes I_{L^p(M)})\|_{\calL(L^{2}(\mathbb{R}^{2};L^p(M)))}$.

Our approach works, more generally, with $L^p(M)$ replaced by a
UMD Banach lattice $X$. The theory of UMD spaces and Banach lattices is presented in \cite{HNVW1} and \cite{mn}, respectively.
By a celebrated theorem of Rubio de Francia \cite{rubio}, every UMD lattice $X$ is a complex interpolation space
between a Hilbert space $H$ and a UMD space $Y$.  In the example $X = L^p(M)$ discussed above one could take $H = L^2(M)$ and $Y = L^q(M)$. Given a self-adjoint operator $L$ on the Hilbert space $H$ satisfying analogues of the above conditions \eqref{it:conditions-Lq1}--\eqref{it:conditions-Lq3}  we prove spectral multiplier results, that is, we prove estimates of the form $\|f(L)u\|_{X} \lesssim \n u\n_X$ for $u \in H\cap Y$, where
$f: [0,\infty) \to \C$ is a bounded measurable function, and $f(L)$ is defined through the spectral theorem as a bounded linear operator on $H$. Since $H\cap Y$ is dense in $X$, this gives estimates for the operator norms $\|f(L)\|_{\mathcal{L}(X)}$. For the special case when $X = L^{p}(M,w)$, with $M$ a doubling metric measure space and $w \in A^{p}$ a Muckenhoupt weight, powerful results of this type, generalising a myriad of special cases, have been obtained in \cite{dsy}. In the present paper we do not need to rely on a doubling measure assumption; instead, we consider operators that have a specific algebraic structure coming from the Heisenberg group.

To start with,  we consider the $H^{\infty}$ class of functions $f: (0,\infty) \to \C$ that have a bounded holomorphic extension to a sector of the complex plane. The theory of the functional calculus associated with such functions is presented in \cite{HNVW2,  KunWei}. In applications, it is often important to have a functional calculus for a wider class of functions which includes compactly supported functions and provides estimates similar to those given by classical spectral multiplier theorems. A H\"ormander calculus enjoying these properties can be obtained by abstract operator theoretic method from the $H^{\infty}$ calculus, provided some refined estimates can be proven for the semigroup generated by $-L$ \cite{KriegW}.

A subtle point in spectral multiplier theory is that results depend on the specific seminorms of
$f$ used to estimate $\|f(L)\|_{\calL(X)}$.
This dependency involves a relationship between integrability and smoothness parameters. See, for example, \cite[Lemma 3.2]{KriegW}, the difference between \cite[Theorem 3.1]{dsy}
and \cite[Theorem 3.2]{dsy}, and the differences between Theorem \ref{thm:main1}, Corollary \ref{cor:Hormander}, and Theorem \ref{thm:main2} in the present paper. The amount of smoothness required in our results is discussed at the end of Remark \ref{rk:BBT}.

In section \ref{sec:Fock}, we give a quintessential example of an operator $L$ to which our results apply: the harmonic oscillator on Bargmann-Fock spaces $A^{p,q}(\C^d)$ (that is, the harmonic oscillator arising from one of the most natural representations of the Heisenberg group). Although not all results obtained in this setting are new (as explained in Remark \ref{rk:BBT}, they can be obtained by combining \cite{BBT} and \cite{FG3}), they are a good showcase for the method: Bargmann-Fock spaces are mixed norm $L^p(L^q)$ spaces with respect to a non-doubling measure and hence not treatable by general harmonic analytic results such as \cite{dsy}.

\section{Main results}
\label{sec:main}
We consider the following setting. Let $X$ be a UMD lattice. By \cite[Corollary, page 216]{rubio}, there exists an interpolation pair $(H,Y)$, with $H$ a Hilbert space, $Y$ a UMD space, and a parameter $\theta \in (0,1)$, such that $X = [Y,H]_{\theta}$, the complex interpolation space of order $\theta$ between $Y$ and $H$.
For later use we recall that $H\cap Y$ is dense in $X$; see, for example, \cite[Corollary C.2.8]{HNVW1}.

 Let $A^H=(A_1^H,\dots,A_d^H)$ and $B^H = (B_1^H,\dots,B_d^H)$ be two $d$-tuples of closed and densely defined operators acting on $H$,
and let $A^Y=(A_1^Y,\dots,A_d^Y)$ and $B^Y = (B_1^Y,\dots,B_d^Y)$ be two $d$-tuples of closed and densely defined operators acting on $Y$,
such that the following hold.
\begin{enumerate}
\item
For all $j=1,\dots,d$ the operators $A_{j}^H$ and $B_{j}^H$ are self-adjoint on $H$.

\item
For all $j=1,\dots,d$ the operators $iA_j^H$ and  $iB_j^H$ generate uniformly
bounded $C_0$-groups $(e^{itA_j^H})_{t\in \R}$
and $(e^{itB_j^H})_{t\in \R}$, respectively, on $H$. Likewise, for all $j=1,\dots,d$
the operators $iA_j^Y$ and  $iB_j^Y$ generate uniformly
bounded $C_0$-groups $(e^{itA_j^Y})_{t\in \R}$
and $(e^{itB_j^Y})_{t\in \R}$, respectively, on $Y$.

\item\label{it:consistent}
For all $j=1,\dots,d$, $t\in \R$, and $x\in H\cap Y$
one has $e^{itA_j^H}x = e^{itA_j^Y}x$ and $e^{itB_j^H}x = e^{itB_j^Y}x$.

\item\label{it:WeylCCR} The {\em Weyl commutation relations} hold both in $H$ and $Y$: for all $j,k=1,\dots,d$ and $s,t\in \R$,
for $Z\in \{H,Y\}$ we have
\begin{equation*}
\begin{aligned}
e^{isA_j^Z}e^{itA_k^Z} &= e^{itA_k^Z}e^{isA_j^Z}, \quad e^{isB_j^Z}e^{itB_k^Z} = e^{itB_k^Z}e^{isB_j^Z},\\
      e^{isA_j^Z}e^{itB_k^Z} & = e^{-ist \delta_{jk}} e^{itB_k^Z}e^{isA_j^Z},
\end{aligned}
\end{equation*}
where $\delta_{jk}$ is the usual Kronecker symbol.
\end{enumerate}
Condition \eqref{it:consistent} states that both the groups $(e^{itA_j^H})_{t\in \R}$ and $(e^{itA_j^Y})_{t\in\R}$, and the groups $(e^{itB_j^H})_{t\in \R}$ and $(e^{itB_j^Y})_{t\in\R}$, are consistent on $H$ and $Y$ in the sense of \cite{tom}. By complex interpolation we obtain bounded $C_0$-groups
$(e^{itA_j})_{t\in \R}$ and $(e^{itB_j})_{t\in\R}$ on $X$. On the dense subspace $H\cap Y$ of $X$  these groups agree with the ones on $H$ and $Y$,
and hence they satisfy the same Weyl commutation relations as in condition \eqref{it:WeylCCR}.

In what follows we will omit the superscripts $H$ and $Y$; it will be clear from the context in which spaces the various groups are considered.
\color{black}

Motivated by  the Baker--Campbell-Hausdorff formula,  in $H$, $Y$ and $X$ we define, for $x,\xi  \in \R^d$,
\begin{align*} \exp(i(xA +\xi B)) &:= \exp\Bigl(\frac12ix_j\xi _k\delta_{jk}\Bigr) \exp(i\sum_{j=1}^d x_jA_j\Bigr)
\exp\Bigl(i\sum_{j=1}^d \xi _jB_j\Bigr) \\ &\phantom{:}=\exp\Bigl(-\frac12ix_j\xi _k\delta_{jk}\Bigr) \exp\Bigl(i\sum_{j=1}^d \xi _jB_j\Bigr)
\exp\Bigl(i\sum_{j=1}^d x_jA_j\Bigr).
\end{align*}
Here we write $xA = \sum_{j=1}^d x_j A_j$ and $\xi B = \sum_{j=1}^d \xi _j B_j$; likewise we will write $x\xi  = \sum_{j=1}^d x_j \xi _j$ and $A^2= \sum_{j=1}^d A_j^2$ and $B^2 = \sum_{j=1}^d B_j^2$.

As is shown in \cite[Proposition 3.13]{NP2}, the operators
$$T_{x,\xi }(t):=   \exp(it(xA +\xi B)), \quad t\in\R,$$
define uniformly bounded $C_0$-groups on  $H$, $Y$, and $X$, and on each of these spaces the intersected domains
$\Dom(A)\cap \Dom(B)$ form a core for its generator $G_{x,\xi }$. Moreover, on this core,
the generator is given by
\begin{align}\label{eq:BCH2}  G_{x,\xi }f  = ix A f+i\xi  Bf, \quad f\in\Dom(A)\cap \Dom(B).
\end{align}
This provides a rigorous interpretation of the operator $uA+vB$ as a closed and densely defined operator.

As proven in \cite[Theorem 5.2]{NP2}, the operator $$-L := -\frac{1}{2}(A^{2}+B^{2})$$ with domain $\Dom(L) = \bigcap_{j=1} ^{d} \Dom(A_{j}^{2}) \cap \Dom(B_{j}^{2})$ generates a uniformly bounded  $C_{0}$-semigroup $(e^{-tL})_{t\ge 0}$ on each of the spaces $H$, $Y$, and $X$.

We will use the notation introduced in \eqref{eq:wtL} and \eqref{eq:wtPQ}.
\begin{lemma}
\label{lem:transfer}
For every bounded, measurable, compactly supported function $f:[0,\infty)\to \C$ there exists a Schwartz function $g \in \mathcal{S}(\R^{2d})$ such that
for the operators
$f({L})$ and $f(\widetilde{L})$, considered as operators on $H$ and { $L^2(\R^{2d})$}, respectively, we have
\begin{equation}\label{eq:fLftL}
\begin{aligned}
f(L) = \frac{1}{(2\pi)^{d}} \int _{\R^{2d}} \widehat{g}(x,\xi) \exp(i(xA +\xi B)) \ud x \ud \xi ,\\
f(\widetilde{L}) = \frac{1}{(2\pi)^{d}} \int _{\R^{2d}} \widehat{g}(x,\xi)\exp(i(x\wt Q +\xi \wt P)) \ud x \ud \xi ,
\end{aligned}
\end{equation}
where
$$  \widehat{g}(x,\xi) := \frac1{(2\pi)^d} \int_{\R^{2d}} g(u,v)\exp(-i(ux +v\xi))\ud u\ud v$$
is the Fourier transform of $g$.
The first identity of \eqref{eq:fLftL} also holds when $A$, $B$, and $L$ are considered as (tuples of) operators acting in $X$.
\end{lemma}

\begin{proof}
We consider $f(L)$ as acting on $H$ and $f({ L_{HO}})$ as acting on $L^{2}(\R^{d})$, where  ${ L_{HO}} = -\frac{1}{2}\Delta + \frac{1}{2}|x|^{2}$.
By a result of Peetre \cite{p} (see also Thangavelu's book \cite[Theorem 1.36]{tbook}) we have
$$f({ L_{HO}}) =
\frac{1}{(2\pi)^{d}} \int _{\R^{2d}} \widehat{g}(x,\xi)
\exp(i(xQ+ \xi P)) \ud x \ud \xi, $$
where $Q = (Q_{1},\dots,Q_{d})$ and $P = (P_{1},\dots,P_{d})$ are the $d$-tuples of { standard position and momentum} operators on $L^2(\R^d)$
 given by
\begin{align*}
Q_{j}f(x):=x_jf(x), \quad P_{j}f(x):=\frac{1}{i}\frac{\partial f}{\partial x_j}(x); \quad x\in\R^d,
\end{align*}
and the Schwartz function $g \in \mathcal{S}(\R^{2d})$ is given by
$\widehat{g} = \sum  _{n \in \N} f(n) L_{n}$, where
$$
L_{n}(x,\xi) = L^{(d-1)}_{n}(|(x,\xi)|^{2}/2)\exp(-|(x,\xi)|^{2}/{4})
$$
with $L^{(d-1)}_{n}$ the $n$-th Laguerre polynomial of type $d-1$ (see \cite[pages 7 and 19-22]{tbook}).
The result now follows from the Stone--von Neumann uniqueness theorem
(see \cite[Theorem 14.8]{Hall}). Indeed, the Hilbert space $H$ can be written as a countable orthogonal direct sum of Hilbert spaces $(H_{k})_{k \in K}$, and there are unitary operators $U_{k}: H_{k} \to L^{2}(\mathbb{R}^{d})$ that establish a unitary equivalence between $\exp(i(xA+ \xi B))_{|H_{k}}$ and
$\exp(i(xQ+ \xi P))$, as well as $f(L)_{|H_{k}}$ and $f(\widetilde{L})$. We thus have, for all $k\in K$,
\begin{align*}
f(L)_{|H_{k}} = U_{k}^{-1} f(\widetilde{L}) U_{k} &= \frac{1}{(2\pi)^{d}} \int  _{\R^{2d}} \widehat{g}(x,\xi)U_{k}^{-1}\exp(i(xQ+ \xi P))U_{k} \ud x \ud \xi \\
&= \frac{1}{(2\pi)^{d}} \int  _{\R^{2d}} \widehat{g}(x,\xi) \exp(i(xA +\xi B))_{|H_{k}} \ud x \ud \xi.
\end{align*}
This gives the first identity in \eqref{eq:fLftL}. The second identity is proven in the same way.
The final assertion follows from the fact that $H\cap X$ is dense in $X$; this is a trivial consequence of the fact that
$H\cap Y$ is contained in $X$ and is dense in this space.
\end{proof}

\begin{remark}\label{rk:invert}
Lemma \ref{lem:transfer} is also valid for certain functions that are not compactly supported.
For instance, \cite[Theorem 5.2]{NP2} gives that
$$
\exp\Bigl(-t(L- \frac{d}{2})\Bigr) = \frac{1}{(2\pi)^{d}} \int  _{\R^{2d}} \widehat{a_{t}}(x,\xi)
\exp(i(xA +\xi B)) \ud x \ud \xi ,
$$
for $a_{t}(x,\xi):=  (1+\lambda_t)^d \exp(-\lambda_t(|x|^2+|\xi|^2))$
with $\lambda_t =  \frac{1-e^{-t}}{1+e^{-t}}$.
In particular, for the harmonic oscillator ${ L_{HO}}=-\frac{1}{2}\Delta + \frac{1}{2}|x|^{2}$ this gives
$$
\exp\Bigl(-t({ L_{HO}}- \frac{d}{2})\Bigr) = \frac{1}{(2\pi)^{d}} \int  _{\R^{2d}} \widehat{a_{t}}(x,\xi)
\exp(i(xQ +\xi P)) \ud x \ud \xi.
$$

The Laguerre transform formula of Peetre used in the proof of Lemma \ref{lem:transfer} sheds  some light on this formula.
Indeed, assuming $d=1$ for simplicity, one has
$$
 \exp\Bigl(-t\bigl({ L_{HO}}-\frac12\bigr)\Bigr)u = \sum _{n \in \N} \exp(-nt)P_{n}u, \quad u \in L^{2}(\R),
$$
where $P_{n}u =  \langle u,H_{n} \rangle H_{n} $ and $H_{n}$ is the $n$-th normalised Hermite function (see \cite[Pages 1-6]{tbook}). By \cite[Theorem 1.3.6]{tbook}, we have
$$P_{n} =  \frac{1}{2\pi} \int _{\R^{2}} L_{n}(\frac{1}{2}(|x|^{2}+|\xi|^{2}))\exp(-\frac{1}{4}(|x|^{2}+|\xi|^{2}))\exp(i(xQ +\xi P)) \ud x \ud \xi.$$
This leads to the identity
\begin{align*}
  & \int _{\R^{2d}} \widehat{a_{t}}(x,\xi)
\exp(i(xQ +\xi P)) \ud x \ud \xi
\\ & = \sum _{n \in \N} \exp(-nt) \int _{\R^{2}} L_{n}(\frac{1}{2}(|x|^{2}+|\xi|^{2}))\exp(-\frac{1}{4}(|x|^{2}+|\xi|^{2}))\exp(i(xQ +\xi P)) \ud x \ud \xi.
\end{align*}
Hence,
by the inversion formula for the Weyl transform \cite[Theorem 1.2.1]{tbook},
\begin{align}\label{eq:Laguerre}\widehat{a_{t}}(x,\xi) =  \sum _{n \in \N} \exp(-nt) L_{n}(\frac{1}{2}(|x|^{2}+|\xi|^{2}))\exp(-\frac{1}{4}(|x|^{2}+|\xi|^{2})).
\end{align}
On the other hand, for all $a>-1$ and $r\ge 0$ one has
$$\exp(-ar) = \sum _{n \in \N} \frac{a^{n}}{(a+1)^{n+1}}L_{n}(r) = \frac1{a+1} \sum _{n \in \N} \frac{1}{(1+a^{-1})^{n}}L_{n}(r)$$
(see, e.g., \cite{erd}). For $a^{-1} = e^{t}-1$ this gives
$$\exp\Bigl(-( e^{t}-1)^{-1} r\Bigr) =  (1-e^{-t})\sum _{n \in \N} \exp(-nt)L_{n}(r).$$
Taking $r = \frac{1}{2}(|x|^{2}+|\xi|^{2})
$ and substituting the resulting identity into \eqref{eq:Laguerre}, after some computations we obtain
\begin{align*} \widehat{a_{t}}(x,\xi)
= \frac{1}{1-e^{-t}}\exp\Bigl(-\frac14 \frac{1+e^{-t}}{1-e^{-t}}(|x|^{2}+|\xi|^{2})\Bigr)
\end{align*}
and hence, inverting the Fourier transform, for $d=1$ we arrive at
$$ a_{t}(x,\xi) = \Bigl(1+ \frac{1-e^{-t}}{1+ e^{-t}}\Bigr)\exp\Bigl(- \frac{1-e^{-t}}{1+e^{-t}}(|x|^{2}+|\xi|^{2})\Bigr)$$
as desired.
\end{remark}

\begin{theorem}
\label{thm:main1}
Let $s>d$ and $f:[0,\infty)\to \C$
be a bounded, measurable, compactly supported function such that
$$\underset{t>0}{\sup} \|\eta\cdot  \delta_{t}(f)\|_{W_{s} ^{\infty}(0,\infty)}<\infty,$$ where $(\delta_{t}(f))(x) = f(tx)$ and $\eta \in C^{\infty}_{\rm c}(0,\infty)$ is a fixed non-zero cut-off function. Then $f(L)$ is bounded on $X$ and we have the estimate $$
\|f(L)\|_{\calL(X)} \leq C \Bigl(\underset{t>0}{\sup} \|\eta\cdot  \delta_{t}(f)\|_{W_{s} ^{\infty}(0,\infty)} + |f(0)|\Bigr),
$$
for a constant $C$ independent of $\eta$ and $f$. \end{theorem}

\begin{proof}
By Lemma \ref{lem:transfer}, we can apply the transference principle \cite[Proposition 6.3]{NP2} to obtain that  $f(L)$ is bounded on $X$ and
\begin{align}\label{eq:transf}
\|f(L)\|_{\calL(X)} \lesssim \|f(\widetilde{L}) \otimes I_{X}\|_{\calL(L^{2}(\mathbb{R}^{2d};X))}.
\end{align}
By Rubio de Francia's extrapolation theorem from \cite[Theorem 5]{rubio} (see also \cite[Theorem 1.1]{alv}), we have
$$\|f(\widetilde{L}) \otimes I_{X}\|_{\calL(L^{2}(\mathbb{R}^{2d};X))}
 \lesssim \Bigl(\underset{t>0}{\sup} \|\eta\cdot  \delta_{t}(f)\|_{W_{s} ^{\infty}(0,\infty)} + |f(0)|\Bigr)$$
as long as
\begin{equation*}
\|f(\widetilde{L})\|_{\calL(L^{2}(\mathbb{R}^{2d},w))}
 \lesssim \Bigl(\underset{t>0}{\sup} \|\eta\cdot  \delta_{t}(f)\|_{W_{s} ^{\infty}(0,\infty)} + |f(0)|\Bigr),
 \end{equation*}
 for all Muckenhoupt weights $w \in A_{2}$. Such {{weighted}} bounds hold for spectral multipliers $f(A)$ as long as $A$ generates a semigroup with pointwise Gaussian heat kernel bounds, thanks to
\cite[Theorem 3.2]{dsy} (where $D=0$ as we work on $\R^{2d}$). We thus only need to prove these heat kernel bounds \cite[Assumption (GE)]{dsy}, i.e., to show that there exist $c,C>0$ such that for all $t>0$ and $x,y,\xi,\eta \in \R^{d}$,
$$
|k_{t}(x,y,\xi,\eta)| \leq Ct^{-d} \exp \Big(-\frac{|x-\xi|^{2}+|y-\eta|^{2}}{ct}\Big),
$$
where $\exp(-t\tilde{L})f(x,y) = \int _{\R^{2d}} k_{t}(x,y,\xi,\eta)f(\xi,\eta)\ud\xi \ud\eta$,
This follows from (5.2) and Lemma 6.4 in \cite{NP2} (or from Remark \ref{rk:invert}), which give that, for all $u \in L^{p}(\R^{2d})$,
$$
\exp\Bigl( -t(\widetilde{L} -\frac12 d)\Bigr)u(x,y) =  \frac{1}{(2\pi)^{d}}\int _{\R^{2d}} \widehat{a_{t}}(x-\xi, y-\eta) \exp(\frac{i}{2}(\xi y-\eta x))u(\xi,\eta) \ud \xi  \ud\eta,
$$
where, as in Remark \ref{rk:invert},
$$a_{t}(x,y) =  (1+\lambda_t)^d\exp\Bigl(-\lambda_t(|x|^2+|y|^2)\Bigr)$$ and
$\lambda_{t} = \frac{1-e^{-t}}{1+e^{-t}}$,
and thus
\begin{align*}
\ & |\exp(-\frac{td}{2}) \widehat{a_{t}}(x-\xi, y-\eta) \exp(\frac{i}{2}(\xi y-\eta x))|
\\ & \qquad\qquad \lesssim \exp(-\frac{td}{2}) \lambda_{t}^{-d} \exp\Bigl(-\frac{|x-\xi|^2+|y-\eta|^2}{\lambda_t}\Bigr).
\end{align*}
\end{proof}

\begin{remark} The transference estimate \eqref{eq:transf}
more generally holds in the form
$$
\|f(L)u\|_{X} \lesssim \|f(\widetilde{L}) \otimes I_{X}\|_{\calL(L^{p}(\mathbb{R}^{2d};X))} \|u\|_{X}.
$$
In the most common situation where $X=L^{p}(M)$ with $p \in (1,\infty)$,  by working with this estimate one can avoid Rubio de Francia's extrapolation theorem. Indeed, a simple application of Fubini's theorem  (see \cite[Proposition 2.2]{HNVW1}) now gives
\begin{align*}
\|f(L)u\|_{L^{p}(M)} &\lesssim \|f(\widetilde{L}) \otimes I_{L^{p}(M)}\|_{\calL(L^{p}(\mathbb{R}^{2d};L^{p}(M)))} \|u\|_{L^{p}(M)}
\\ & = \|f(\widetilde{L})\|_{\calL(L^{p}(\mathbb{R}^{2d}))} \|u\|_{L^{p}(M)}.
\end{align*}
The $L^{p}$ estimates for $f(\widetilde{L})$ then follow from \cite[Theorem 3.2]{dsy} (or even older results without weights).
\end{remark}

\begin{theorem}
\label{thm:Hinf}
The operator $L-\frac{d}{2}I$
has a bounded $H^{\infty}$ functional calculus of zero angle {on $X$}.
\end{theorem}

\begin{proof}
The proof uses the $R$-sectoriality of $L-\frac{d}{2}I$  proven in \cite[Theorem 7.1]{NP2} and the abstract square function characterisation for boundedness of the  $H^{\infty}$ calculus (see \cite[Theorem 10.4.9]{HNVW2}) exactly as in the proof of \cite[Theorem 8.1]{NP2}, up to the crucial estimate
\begin{equation*}
\sup_{s\in [1,2]}\sup_{N\ge 1}\mathbb{E}\Big\|\sum_{j=1}^{N} \eps_{j} \wt b_{2^{-j}s}(A,B)u\Big\|^2\lesssim \n u\n^2, \quad  u \in X,
\end{equation*}
where $\wt b_{t} : = (1+\lambda_{2t})^{-d}a_{2t} - (1+\lambda_{t})^{-d}a_{t}$. In \cite{NP2}, this estimate is proven using an {\em a priori} assumption on the $S^{0}$-boundedness of the Weyl calculus of the Weyl pair $(A,B)$. Here, we use transference as in the proof of Theorem \ref{thm:main1} instead. Let $s \in [1,2]$, $N \in \N\setminus\{0\}$, and $(\varepsilon_{j})_{j=1} ^{N} \in \{\pm 1\}^{N}$.  By the transference principle \cite[Proposition 6.3]{NP2} and \cite[Theorem 5]{rubio} (see also \cite[Theorem 1.1]{alv}), it suffices to show that
\begin{align*}
\Bigl\|\sum_{j=1}^{N} \eps_{j} & \Bigl((1+\lambda_{2^{-j+1}s})^{-d}\exp(-2^{-j+1}s\widetilde{L}) \\ & \qquad - (1+\lambda_{2^{-j}s})^{-d}\exp(-2^{-j+1}s\widetilde{L}) \Bigr)\Bigr\|_{\calL(L^{2}(\R^{2d},w))} \leq C,
\end{align*}
for all Muckenhoupt weights $w \in A_{2}$, and a constant $C\ge 0$ independent of $s,N$ and $(\varepsilon_{j})_{j=1} ^{N}$. This follows from \cite[Proposition 6.8]{dsy}, and the fact that
$$
z \mapsto \sum_{j=1}^{N} \eps_{j} \left((1+\lambda_{2^{-j+1}z})^{-d}\exp(-2^{-j+1}z) - (1+\lambda_{2^{-j}z})^{-d}\exp(-2^{-j+1}z) \right)
$$
belongs to $H^{\infty}(\Sigma_{\omega}):= \{z\in\C\setminus\{0\}:\, |\arg(z)|<\om\}$ for all $\omega\in (0,\frac12\pi)$, with a norm independent of $N$ and $(\varepsilon_{j})_{j=1} ^{N}$,
by \cite[Proposition H.2.3]{HNVW2}.
\end{proof}

With the same proof as \cite[Theorem 8.5]{NP2}, using Theorem \ref{thm:Hinf} and \cite{KriegW}, we can upgrade the holomorphic calculus to a H\"ormander calculus.

\begin{corollary}\label{cor:Hor}
The operator $L-\frac{d}{2}I$ has an $R$-bounded $\mathcal{H}_{2} ^{2d+\frac{1}{2}}$-H\"ormander calculus {on $X$}.
\end{corollary}

As proven recently in \cite[Theorem 1.1]{DK}, this calculus can be further upgraded to give the following maximal function estimates.

\begin{corollary}
\label{cor:Hormander}
Assume that $X = L^{p}(\Omega;Z)$, with $p \in (1,\infty)$ and $Z$ a UMD lattice, and let $\Omega$ a $\sigma$-finite measure space. Let
$s> (2d+\frac32) + \max\{\frac12, \tau_{p,Z}^{-1} - c_{p,Z}^{-1}\}$, where $\tau_{p,Z}$ and $c_{p,Z}$ are the type and cotype of $L^p(\Omega;Z)$,
 and let $f \in W_{s} ^{2}(0,\infty)$  satisfy $f(0)=0$ and
$$\sum _{n \in \Z} \|\eta\cdot \delta_{2^{n}}(f)\|_{W_{s} ^{2}(0,\infty)}<\infty,$$
where $(\delta_{t}(f))(x) = f(tx)$,  $\eta \in C^{\infty}_{\rm c}(0,\infty)$ is a fixed cut-off function constantly equal to $1$ on the interval $(1,2)$. Then
$$
\Bigl\|\underset{t>0}{\sup}|f(tL)u|\Bigr\|_{L^{p}(\Omega;Z)} \leq
C \sum _{n \in \Z} \|\eta\cdot \delta_{2^{n}}(f)\|_{W_{s} ^{2}(\R)}\|u\|_{L^{p}(\Omega;Z)}, \quad  u \in L^{p}(\Omega;Z),
$$
for a constant $C$ independent of $\eta$ and $f$.
\end{corollary}

For compactly supported symbols, we can also obtain the following H\"ormander type spectral multiplier estimates under less stringent assumptions on the regularity parameter $s$. Recall that $X = [Y,H]_{\theta}$ for some UMD space $Y$. For instance, when $X=L^{p}(M)$ for some arbitrary measure space $M$ (not necessarily doubling), we
may take $H = L^2(M)$ and $Y = L^q(M)$ with $1<q<p$ (when $1<p< 2$), respectively $p<q<\infty$ (when $2<p<\infty$).

\begin{theorem}
\label{thm:main2}
Let $s> (1-\theta)d+\frac{1}{2}$ and let $f:[0,\infty)\to \C$ be a bounded, measurable, compactly supported function such that
$$\underset{t>0}{\sup} \|\eta\cdot  \delta_{t}(f)\|_{W_{s} ^{2}(0,\infty)}<\infty,$$ where $(\delta_{t}(f))(x)= f(tx)$ and $\eta \in C^{\infty}_{\rm c}([0,\infty))$ is a fixed non-zero cut-off function. Then we have the estimate
$$
\|f(L)\|_{\calL(X)} \leq C \Bigl(\underset{t>0}{\sup} \|\eta\cdot \delta_{t}(f)\|_{W_{s} ^{2}(0,\infty)} + |f(0)|\Bigr),
$$
for a constant $C$ independent of $\eta$ and $f$.
\end{theorem}

\begin{proof}
Applying Theorem \ref{thm:Hinf} to $\widetilde{L} \otimes I_{X}$, we have $\widetilde{L} \otimes I_{X}$ has a bounded $H^{\infty}$ calculus in $L^{2}(\R^{d};X)$. By \cite[Theorem 1.2]{dkk}, we thus have the result for $\widetilde{L} \otimes I_{X}$ acting on $L^{2}(\R^{d};X)$ instead of $L$ acting on $X$. An application of Lemma \ref{lem:transfer}  (this is where the compact support assumption is used) and the transference theorem \cite[Proposition 6.3]{NP2} complete the proof.
\end{proof}

\section{Application to Bargmann-Fock space}
\label{sec:Fock}

In this section we show that our result can be applied to the harmonic oscillator on the Bargmann-Fock spaces of entire functions in $\mathbb{C}^d$ defined by
\begin{equation}
    A^{p,q}(\mathbb{C}^d)=\Big\{f \;\text{entire}: \int_{\R^{d}}\Big(\int_{\R^{d}}|f(x+iy)|^p e^{-p|x+iy|^2/2}\ud x\Big)^{q/p}\ud y<\infty\Big\},
\end{equation}
for $1<p,q<\infty$.
When $p=q=2$, $A^{2,2}(\mathbb{C}^{d})$ is the Hilbert space used for the Bargmann representation of the Heisenberg group. We recall its construction below. This example is  natural from a representation theory/mathematical physics point of view, but is quite challenging from a harmonic analysis point of view, because the measure $e^{-p|x+iy|^2/2}\ud x \ud y$ is not doubling (and hence results such as \cite{dsy} do not apply).

One defines the {\it raising} and {\it lowering operators} $a_j^*$ and $a_j$  by
\begin{align*}
    a_j^* f(z):=z_j f(z), \quad
    a_j f(z):=\frac{\partial}{\partial z_j}f(z), \quad  f \in A^{2,2}(\C^{d}),
\end{align*}
and the  {\it position} and {\it momentum} operators as follows:
\begin{equation}\label{P&M_onBarg}
 \begin{split}
    A_j&=\frac{1}{\sqrt{2}}(a_j^*+a_j),\\
    B_j&=\frac{i}{\sqrt{2}}(a_j^*-a_j).
 \end{split}
\end{equation}
It is easy to check that the operators $A_j$ and $B_j$ are self-adjoint, and that the pair $(A,B)$, where $A=(A_1,A_2,\dots,A_d)$ and $B=(B_1,B_2,\dots,B_d)$, satisfies the canonical commutations relations. We shall now see that $iA_j$ and $iB_j$ generate uniformly bounded $C_0$-groups on $A^{p,q}(\C^d)$, for all $1<p,q<\infty$.

For $a\in\C^d$, define operators $T_a$ by the formula
\begin{equation}\label{OpTa}
    (T_af)(z)=e^{\frac{-|a|^2}{2}}e^{-\bar{a}\cdot z}f(z+a)
\end{equation}
where $a\cdot b=\sum_ja_jb_j$, for any $a,b\in\C^d$.
\begin{lemma}
\label{lem:iso}
Let $1<p,q<\infty$.
For each $a\in\C^d$, the operator $T_a$ defined above is an isometry of $A^{p,q}(\C^d)$, and the map $a\mapsto T_a$ is strongly continuous.
\end{lemma}
\begin{proof}
For $f\in A^{p,q}(\C^d)$ we have (cf. \cite[Theorem 14.16]{Hall} for the special case $p=q=2$)
\begin{align*}
    ||T_af||^q_{A^{p,q}}&=\int_{\R^d}\Big(\int_{\R^d}e^{\frac{-p|a|^2}{2}}e^{-p\Re(\bar{a}\cdot z)}|f(x+iy+a)|^p e^{\frac{- p|z|^2}{2}}\ud x\Big)^{q/p}\ud y\\
    &=\int_{\R^d}\Big(\int_{\R^d}e^{-\frac{p}{2}|z+a|^2}|f(x+iy+a)|^p\ud x \Big)^{q/p}\ud y\\
    &=||f||^q_{A^{p,q}}.
\end{align*}
It can also be verified easily that $T_aT_b=e^{i \Im(\bar{a}\cdot b)}T_{a+b}$, which implies $T_aT_{-a}=I$. The strong continuity of $T_a$ can be verified on the set of polynomials, which are dense in $A^{p,q}(\C^d)$ (see \cite[Proposition 5]{GWo}).
\end{proof}

\begin{proposition}
Let $1<p,q<\infty$. The tuples of operators $A=(A_{j})_{j=1} ^{d}$ and $B=(B_{j})_{j=1} ^{d}$ form a Weyl pair in $A^{p,q}(\C^{d})$.
\end{proposition}
\begin{proof}

For $j=1,\dots,d$, define $U_j(t):=T_{ite_j/\sqrt{2}}$ and $V_j(t):=T_{te_j/\sqrt{2}}$. By Lemma \ref{lem:iso}, we have that  $\{U_j(t)\}_{t\in\mathbb{R}}$ and $\{V_j(t)\}_{t\in\mathbb{R}}$ are uniformly bounded $C_0$-groups in $A^{p,q}(\C^{d})$. The infinitesimal generator of these groups are $iA_j$ and $iB_j$ respectively. Indeed, if $f$ lies in the domain of the infinitesimal generator of $U_j(\cdot)$, the limit
\begin{equation*}
    (iA_jf)(z):=\lim_{t\to0}\frac{1}{t}[e^{t^2/{4}}e^{itz_j/\sqrt{2}}f(z+ite_j/{\sqrt{2}})-f(z)]
\end{equation*}
must exist in $A^{p,q}(\C^d)$ and coincide with the pointwise limit giving,
\begin{equation*}
    iA_jf(z)=\frac{i}{\sqrt{2}}\Big(z_jf(z)+\frac{\partial}{\partial z_j}f(z)\Big)
\end{equation*}
By dominated convergence theorem, the limit exists in $A^{p,q}$ for polynomials and hence by the density argument, for functions in the domain of $A$. The same analysis work for $B_j$ as well. Thus, $(A,B)$ forms a Weyl pair on $A^{p,q}(\C^d)$.
\end{proof}
All the spectral multiplier results proven in Section \ref{sec:main} thus apply to the harmonic oscillator $L = \frac{1}{2}\sum _{j=1} ^{d}(A_{j}^{2}+B_{j}^{2})$ acting on $A^{p,q}(\C^{d})$.

\begin{remark}
\label{rk:BBT}
It is shown in \cite{FG3} that $A^{p,q}$ can be identified with the modulation space $M^{p,q}$.
This identifies $L$ with the standard harmonic oscillator $L^{(d)}$ for which spectral multiplier results have been proven in \cite{BBT}.
For $p \neq q$ Theorem \ref{thm:main2} recovers \cite[Theorem 3.1]{BBT}. However, for $p=q$,  \cite[Theorem 3.3]{BBT} is a better result, as it only requires the optimal level of regularity $s>d/2$. The present paper and \cite{BBT} share a similar philosophy: to use transference as a replacement for the Stone--von Neumann uniqueness theorem when proving $L^p$ estimates for representations of the Heisenberg group. The paper \cite{BBT} focuses on concrete representations to reduce the problem to Fourier multipliers on the Heisenberg group. Here we take a more general approach that allows us to treat any representation on any UMD lattice. At this level of generality, we do not reach the optimal order of smoothness one could hope for in our assumptions. This is to be expected because our transference principle doubles the dimension (for example, if $X=L^{p}(\R^{d})$, we transfer to $L^{2}(\R^{2d};L^{p}(\R^{d})$). However, it is interesting to note that, even for the standard harmonic oscillator on $M^{p,q}$ for $p \neq q$, the result in \cite{BBT} features the exact same loss in the order of smoothness as in our Theorem \ref{thm:main2} ($s>d+\frac{1}{2}$ rather than $s>d/2$). One can also note that Corollary \ref{cor:Hormander} gives maximal function estimates for the harmonic oscillator on modulation spaces (or on Bargmann-Fock spaces). Such maximal function estimates appear to be new, and could be useful in PDE applications, in the spirit of the recent paper \cite{bmntt}.
\end{remark}

\section{open problems}

In this paper, we have partially solved the open problems from \cite{NP2}, by proving appropriate functional calculi results for $L$ under the sole assumption that $(A,B)$ is a Weyl pair on a UMD lattice. This leaves open many interesting questions that can be divided into four categories:
\begin{enumerate}
\item
Extending this transference approach to other nilpotent Lie groups (beyond the Heisenberg group).
\item Treating more general UMD spaces (in particular non-commutative $L^p$ spaces). \item Obtaining a full pseudo-differential calculus rather than just spectral multiplier results for Laplace like operators (that is, a joint functional calculus for $(A,B)$ rather than a functional calculus for $L$).
\item
Determining the optimal order of smoothness required for general spectral multiplier theorems.
\end{enumerate}
In the fourth direction, Remark \ref{rk:BBT} suggests that the order $s>d+\frac{1}{2}$ may not be as bad as it looks, and may even be optimal if one wants to treat any UMD lattice.
In the second direction, Lukas Hagedorn (personal communication)  has made substantial progress
using UMD valued Fourier multipliers directly on the Heisenberg group.
To make progress on the first and second directions, it would be highly interesting to extend the theory of oscillatory singular integral operators from \cite{RS} to the UMD valued setting (as it is done in \cite[$Tb$ Theorem 3]{Hyt-vvTb} for standard singular integral operators) or to the semicommutative $L^p$ setting (as it is done in \cite[Section 4.1]{jmpx} for standard singular integral operators).

\end{document}